\documentclass[a4paper]{amsart}

\usepackage{amssymb}
\usepackage{amsmath,amsthm}
\usepackage{cite}

\usepackage{enumerate,paralist}
\usepackage{amstext}
\usepackage{psfrag}
\usepackage{dsfont}
\usepackage[dvips]{epsfig}
\usepackage[english]{babel}
\usepackage{url}
\usepackage{graphicx}
\usepackage{dcolumn}
\usepackage{bm}
\usepackage{float}
\usepackage{hyperref}
\usepackage{color}
\usepackage{epstopdf}
\usepackage{cleveref}
\usepackage[svgnames]{xcolor}
\hypersetup{hidelinks,colorlinks=true,allcolors=DarkBlue}

\numberwithin{equation}{section}

\theoremstyle{plain}
\newtheorem{theorem}{Theorem}

\theoremstyle{definition}
\newtheorem{definition}{Definition}

\DeclareMathOperator{\sign}{sign}

\def\text#1{\mbox{\rm #1\,}}

\begin{document}
\frenchspacing

\title[Damping of a system of linear oscillators]
{Damping of a system of linear oscillators \\ using the generalized dry friction}

\author{Alexander Ovseevich}
\address
{
Institute for Problems in Mechanics, Russian Academy of Sciences \\
119526, Vernadsky av., 101/1, Moscow, Russia } \email{ovseev@ipmnet.ru}

\author{Aleksey Fedorov}
\address
{
Institute for Problems in Mechanics, Russian Academy of Sciences \\
119526, Vernadsky av., 101/1, Moscow, Russia \& Bauman Moscow State Technical University, 105005, 2nd Baumanskaya str.,
5, Moscow, Russia } \email{akf@rqc.ru}

\maketitle

\begin{flushright}
    {\it  Andrei I. Subbotin \\ in memoriam}
\end{flushright}

\begin{abstract}
The problem of damping a system of linear oscillators is considered. 
The problem is solved by using a control in the form of  dry friction. 
The motion of the system under the control is governed by a system of differential equations with discontinuous right-hand side. 
A uniqueness and continuity theorem is proved for the phase flow of this system. 
Thus, the control in the form of generalized dry friction defines the motion of the system of oscillators uniquely.



\end{abstract}

\section*{Introduction}

The present work is closely related to the talk of the one of the authors (A.I.O.) at the international seminar
 ``Control Theory and Theory of Generalized Solutions of Hamilton--Jacobi Equations'' (CGS'2015) dedicated to the memory of A.I. Subbotin. 
The subject of this work correlates with  one of the central topics of work of Andrei Izmailovich: 
How to define the solution of a problem so that the whole corresponding theory takes an attractive and ultimate form.

Of course, Andrei Izmailovich and we deal with substantially different problems. 
Our work is devoted not to solution of nonlinear partial differential equations but to ordinary differential equations (ODEs) and related linear transport equations. 
Nevertheless, the basic idea of the work consists in obtaining the final existence and uniqueness of the motion for a very specific system of differential equations 
with discontinuous right-hand side (RHS) on the basis of a nonconventional concept of a solution.

In the  case under consideration, the system arises in an attempt to control in a quasioptimal way a system of an arbitrary number of linear oscillators 
by using a feedback control in the form of a generalized dry friction \cite{ovseev0,ovseev1,ovseev2}, 
and the adopted concept of solution is suggested by the classical work by R. DiPerna and P. Lions \cite{diperna}.

\section{Problem statement and background}

As it is well-known, by using the maximum principle one can explicitly construct a feedback control  for the minimum-time damping of a single linear oscillator \cite{pont}. 
A direct generalization of such task is the problem of damping a system consisting of an arbitrary number $N$ of linear oscillators with different eigenfrequencies $\omega_i$:
\begin{equation}\label{syst1}
	\dot{x}={A}x+{B}u,
	\qquad
	x\in\mathbb{V}=\mathbb{R}^{2N},
	\qquad
	u\in\mathbb{U}=\mathbb{R},
	\qquad |u|\leq1,
\end{equation}
where the matrix $A$ and the vector $B$ have the form
\begin{equation}\label{syst2}
	{A_i}=
	\left({\begin{array}{*{20}c}
	0   &   1  \\
	{ - \omega_i^{2}} & 0 \\
	\end{array}}\right), \qquad
	A={\rm diag}(A_i), \qquad
	{B_i} = \left( \begin{gathered}
	0 \hfill \\
	1 \hfill \\
	\end{gathered}\right), \qquad
	B=\oplus B_i.
\end{equation}
In the natural coordinates $(x_i,y_i),\,i=1,\dots,N$, the system can be written as follows:
$$
\begin{array}{l}
	\dot{x_i}=y_i, \\
	\dot{y_i}=-\omega_i^{2}{x_i}+{u}, \quad |u|\leq1, \quad i=1,\dots,N.
\end{array}
$$

Asymptotical properties of the system are  determined primarily by the presence or absence of resonances, i.e., non-trivial relations between eigenfrequencies of the form
\begin{equation}\label{reson}
	\sum_{i=1}^Nm_i\omega_i=0,\, \mbox{ ãäå } 0\neq m=(m_1,\dots,m_N)\in {\mathbb{Z}}^N.
\end{equation}
The Kalman controllability criterion \cite{kalman} for system (\ref{syst1})--(\ref{syst2}) reduces to the distinction of the eigenfrequencies: $\omega_i\neq\omega_j$ at $i\neq j$ and it is, 
of course, substantially weaker than the condition of absence of resonances.

The minimum-time problem for system (\ref{syst1})--(\ref{syst2}) can be reduced to the boundary-value problem of the Pontryagin maximum principle
\begin{equation}\label{max}
\begin{array}{l}
	\dot{x}={A}x+{B}u, \quad \dot{p}=-{A^*}p,\\[.5em]
	u={\rm sign}\langle{B,p}\rangle, \quad x(0)=x_0, \quad x(T)=0, \quad h(x,p)=0,
\end{array}
\end{equation}
with respect to the Hamiltonian
$$
	h(x,p)=\langle{Ax,p}\rangle+|\langle{B,p}\rangle|-1=\max_{|u|\leq1}\{\langle{Ax,p}\rangle+\langle{Bu,p}\rangle-1\},
$$
where angle brackets $\langle{.,.}\rangle$ stand for the standard scalar multiplication in $\mathbb{R}^{2N}$, 
$|\cdot|$ is the Euclidean norm, 
and the maximum is taken over the interval $\{u\in\mathbb{R}:|u|\leq1\}$.

We note that system (\ref{max}) is the Hamiltonian one with $2N$ degrees of freedom and $N+1$ integrals of motion. These
integrals are the Hamiltonian $h$ and energies
$$
	I_i=\frac12(\eta_i^2+{\omega_i^{-2}}{\xi_i^{2}}),\quad i=1,\dots,N
$$
of normal oscillations of the vector $p$ written in the from $p_i=(\xi_i,\eta_i)$, where the variables $\xi_i$ and $\eta_i$ are dual to $x_i$ and $y_i$, respectively. 
From the point of view of the canonical system of the maximum principle (\ref{max}), the problem of the optimal damping of a single linear oscillator is completely integrable, 
because in this case the number $N+1$ of integrals of motions is equal to the number $2N$ of degrees of freedom. 
Note that the similar identity is the basic assumption of the Liouville--Arnold theorem on the complete integrability of Hamiltonian systems \cite{Arnold}. 
In contrast, problem (\ref{syst1})--(\ref{syst2}) is probably not completely integrable, 
so that an analytic design of the optimal control by using methods based on the Pontryagin maximum principle is hardly possible.

The general problem to be solved is to construct a nonoptimal feedback control steering the system to equilibrium. 
An approach to design of an asymptotically optimal feedback control for system (\ref{syst1})--(\ref{syst2}) was suggested in \cite{ovseev0,ovseev1}. 
When using methods of \cite{ovseev0,ovseev1}, 
the ratio of the motion time to zero under the suggested control to the minimum time is close to one provided that the initial energy of the system is sufficiently large.

Within the framework of the present work, 
we confine ourselves with issues related to a construction of the control in the generalized dry friction form and to study of the motion of the system under this control. 
We describe an approach to design of the control, 
study systems of differential equations with a discontinuous right-hand side related to the motion of the system, 
and show that the motion can be defined uniquely. 
In terms of Refs. \cite{ovseev0,ovseev1}, this gives a description of the dynamics of the system under the asymptotically optimal control within the high-energy regions of the phase space. 
In the present work, we provide proofs of existence and uniqueness theorems for arising deferential equations with discontinuous RHS. 
Investigations of asymptotic properties and detailed description of the control can be found in \cite{ovseev1,ovseev2}

It should be noted that substantially different approaches to construction of the feedback control for linear systems are known as well,
for example, the methods based on the Kalman approach to the feedforward control \cite{chern,ovseev+}. 
Papers \cite{chern,ovseev+} include estimates for the motion time under this control. 
The time is comparable with the optimal one: the ratio of the motion time to zero under this control to the minimum one is bounded.

\section{Control for a system of oscillators}

A well-known geometric interpretation of the Hamilton--Jacobi--Bellman equation for the minimum-time problem is that the momentum vector ${\partial T}/{\partial x}$ 
at a point $x$ is an inner normal to the reachable set $\mathcal{D}(T(x))$,
where $T(x)$ is the optimal time for the controllable system.

\begin{definition}
	A reachable set $\mathcal{D}(T)$ is the set of ends of admissible trajectories of a controllable system from zero and parametrized by the time interval $[0.T]$.
\end{definition}

Optimal control has the form
\begin{equation}\label{control_o}
	u(x)=-{\rm sign}\langle B,p(x) \rangle, \qquad p=\frac{\partial T}{\partial x}(x),
\end{equation}
where $p$ is the outer normal to the reachable set $\mathcal{D}(T(x))$ form zero, whose boundary passes through $x$.

The control in the dry friction form
\begin{equation}\label{control0}
	u=-{\rm sign\,}{\sum_{i=1}^{N}\lambda_iy_i},
\end{equation}
where $\lambda_i$ are some positive coefficients, 
depending on positions and velocities, arises when an asymptotic approximation is substituted for the exact reachable set $\mathcal{D}(T(x))$ of system (\ref{syst1})--(\ref{syst2}).
According to the asymptotic theory of reachable sets \cite{ovseev3,ovseev4,ovseev5}, as time $T\to\infty$ $T\Omega$,
where $\Omega$ is a fixed convex body, serves as a good approximation to $\mathcal{D}(T)$. 
The body $\Omega$ can be uniquely defined by its support function.

\begin{definition}
	The support function ${H}_{M}(\xi)$ of a closed convex set $M$ has the form
	\begin{equation}\label{SF}
		{H}_{M}(\xi)=\sup_{x\in M}\langle{\xi, x\rangle}
	\end{equation}
	and defines the set $M$ uniquely \cite{schneider}.
\end{definition}

More precisely, for the system (\ref{syst1})--(\ref{syst2}) under consideration.
\begin{theorem}\label{support0}
	Suppose that a momentum $p$ is written in the form $p=(\xi_i,\eta_i)$,
	where
	$\xi_i$ is the dual variable for $x_i$,
	$\eta_i$ is the dual variable for $y_i$, and
	$z_i=(\eta_i^2+{\omega_i^{-2}}{\xi_i^{2}})^{1/2}$.
	In the nonresonant case, when there are no nontrivial relations between eigenfrequencies of the form (\ref{reson}),
	the support function $H_T$ of the reachable set $\mathcal{D}(T)$ has as $T\to\infty$ the asymptotic form
	\begin{equation}\label{approxN0}
		{H}_T(p)=T\int\limits_{\mathcal{T}}\left|\sum_{i=1}^N z_{i}\cos\varphi_{i}\right|d\varphi+o(T)=	T\mathfrak{H}(z)+o(T),
	\end{equation}
	and the support function of the convex body $\Omega$ is given by the main term $\mathfrak{H}(z)$.
\end{theorem}

\begin{proof}
By definition, the support function of the set $\mathcal{D}(T)$ has the form (\ref{SF}), 
where the supremum is taken over admissible controls and $x(T)$ is the state at time $T$ of the control system (\ref{syst1})--(\ref{syst2}) such that $x(0)=0$.

By using the Cauchy formula, we conclude that
$$
	\langle x(T),p\rangle=\int_0^T \langle e^{A(T-t)}Bu(t),p\rangle dt=\int_0^T u(t)B^*e^{A^*(T-t)}p dt.
$$
Taking the supremum under the integral sign and the change  of variables $t\mapsto T-t$, we obtain
\begin{equation}\label{proof1}
    H_{\mathcal{D}(T)}(p)=\int_0^T \sup_{|u(t)|\leq1} u(t)B^*e^{A^*(T-t)}p dt=\int_0^T |B^*e^{A^*t}p| dt.
\end{equation}
In the dual coordinates $\xi_i,\eta_i$, formula (\ref{proof1}) takes the form
$$
	H_{\mathcal{D}(T)}(p)=\int_0^T \left|\sum_{i=1}^N\eta_i\cos\omega_it+\omega_i^{-1}\xi_i\sin\omega_it\right|dt.
$$
We note that this expression represents an integral of the function
$$
	f(\varphi)=\left|\sum_{i=1}^N\eta_i\cos\varphi_i+\omega_i^{-1}\xi_i\sin\varphi_i\right|
$$
taken over the rectilinear winding $\varphi_i(t)=\omega_it$ of the torus $\mathcal{T}=(\mathbb{R}/2\pi\mathbb{Z})^N$ with angular coordinates $\varphi_i$.

Suppose that there are no resonances in the system, i.e., 
there are no nontrivial relation between eigenfrequencies of the form (\ref{reson}). 
Then the time average coincides with the space average \cite{Arnold}:
$$
	\lim\limits_{T\to\infty}{\frac1T\int_0^T f(\varphi(t))dt}=\int_{\mathcal{T}}f(\varphi)d\varphi
$$

We note that 
$$
	\eta_i\cos\varphi_i+\omega_i^{-1}\xi_i\sin\varphi_i=z_i\cos(\varphi_i+\alpha_i),
$$ 
where $\alpha=(\alpha_i)$ is a constant point of the torus. 
Therefore,
$$
	\int_{\mathcal{T}}f(\varphi)d\varphi=\int_{\mathcal{T}}f(\varphi-\alpha)d\varphi
	=\int_{\mathcal{T}}\left|\sum_{i=1}^Nz_i\cos\varphi_i\right|d\varphi.
$$
Thus, we obtain the statement of the theorem
\begin{equation}
	\lim\limits_{T\to\infty}\frac1T H_{\mathcal{D}(T)}(p)=\int_{\mathcal{T}}\left|\sum_{i=1}^Nz_i\cos\varphi_i\right|d\varphi.
\end{equation}
\end{proof}

Note that the theorem follows from the general theory describing the asymptotic behaviour of support function of reachable sets of linear systems developed in detail in \cite{ovseev4}; 
however the above proof is much simpler than the general theory.

The support function ${H}_\Omega(p)$ of the convex body $\Omega$ is the main term of the asymptotic expression in
(\ref{approxN0})
$$
	{H}_\Omega(p)=\mathfrak{H}(z)=\int\left|\sum_{i=1}^N z_{i}\cos\varphi_{i}\right|d\varphi, \mbox{ where } z=(z_1,\dots,z_N)\in{\mathbb R}^N.
$$
The vector $p$ is  normal to the boundary of $\partial\Omega$ at the point  ${\partial H_\Omega(p)}/{\partial p}$.
The normal vector to the approximate reachable set $\rho\Omega$, whose boundary passes through $x$, is defined by the
equation
\begin{equation}\label{approx30}
    \rho^{-1}x=\frac{\partial {H}_\Omega(p)}{\partial p}=\frac{\partial\mathfrak{H}(z)}{\partial z}\frac{\partial z}{\partial p},
\end{equation}
where $p\in\mathbb{R}^{2N}$ and $\rho>0$ are unknown. 
The function is differentiable, and Eq. (\ref{approx30}) has  unique solution, because the boundary of $\Omega$ is smooth \cite{ovseev5}. 
The strategy of the control design, which follows from Eq. (\ref{approx30}), can be applied in the resonant case as well, when the asymptotic expression (\ref{approxN0}) does not work; 
however in this case quasioptimal properties of the control are lost. 
The function $\rho=\rho(x)$ from Eq. (\ref{approx30}) plays the same role for the considered control as the optimal time $T(x)$ for the optimal control (\ref{control_o}). 
The momentum $p$ in (\ref{approx30}) has the form $p={\partial\rho}/{\partial x}$. 
The function $\rho(x)$ is the norm of the vector $x$ in the metric such that the body $\Omega$ is the unit ball.
This is a smooth function outside zero.

Although using  the control in the dry friction form (\ref{control0}) helps to damp oscillations, 
it does not necessarily lead to a complete stop of the system. 
Strictly speaking, standstill zones might appear, where the system is not moving at all, despite the fact that the equilibrium is not yet reached. 
The method suggested in \cite{ovseev0,ovseev1,ovseev2} combines several strategies of control applied successively at high, medium, and low energies. 
At high and intermediate energies, the scalar control in the form of the generalized dry friction is used (\ref{control0}). 
At low energies, a substantially different feedback control law, which is constructed by using common Lyapunov functions \cite{ovseev0}, is employed. 
In the present work, we are interested in the motion within regions of high and medium energies.

\section{Motion under the control}

Differential equation with discontinuous RHS occur naturally in optimal control theory. 
A conventional approach to the problem of existence of  solutions for such equations is based on the Filippov theory of differential inclusions \cite{filipp}. 
However, an intuitive concept of a controlled motion includes both an existence and a unique determination of the trajectories of the system by a control law. 
The corresponding uniqueness issue is generally beyond the Filippov theory.

The control in the generalized dry friction form also leads to the motion of a system.
This motion is formally described by a differential equation with discontinuous RHS:
\begin{equation}\label{sing_sys}
	\dot x=Ax- B\,{\rm sign}\left\langle{B,\frac{\partial\rho}{\partial x}}\right\rangle,
	\qquad
	u(x)=-{\rm sign\,}\left\langle B,\frac{\partial\rho}{\partial x}\right\rangle.
\end{equation}
Here $u(x)$ is a multivalued function, because $\sign(0)$ is defined non-uniquely and can take any value in the interval $[-1,1]$. 
In fact, we deal with the differential inclusion.

As it turns out, a motion of the system under the control in the generalized dry friction form, which is described by the
differential inclusion (\ref{sing_sys}), can be defined uniquely. Toward this end  one can apply the DiPerna--Lions
theory of singular ODEs \cite{diperna}.

\subsection{DiPerna--Lions theory}

If $b(x)$ is a Lipschitz function, then the Cauchy problem for the ODE
\begin{equation}\label{ODE}
	\dot x=b(x),\quad x(0)=x_0
\end{equation}
and for the partial differential equation (transport equation)
\begin{equation}\label{PDE}
	\frac{\partial v}{\partial t}=\sum b_i(x)\frac{\partial v}{\partial x_i},
	\quad
	v(x,0)=v_0(x)
\end{equation}
are equivalent. 
The method of characteristics says that the solution $v$ of problem (\ref{PDE}) is given by the formula
\begin{equation}\label{flow}
	v(x,t)=v_0(\phi_t(x)),
\end{equation}
where $\phi_t$ is the phase flow for (\ref{ODE}).

In the paper of DiPerna and Lions \cite{diperna}, 
the Lipschitz condition $\partial{b}/\partial{x}\in{L}_\infty$ is substantially relaxed. 
Instead of it, the Lipschitz condition in an integral form $\partial{b}/\partial{x}\in{L}_1$ is imposed. 
It is shown that the solution of problem (\ref{PDE})  still exists, is unique, and is given by formula (\ref{flow}). 
Thereby, is was demonstrated that one can efficiently work with a differential equation whose RHS satisfies the Lipschitz condition in the integral sense rather than pointwise.

The DiPerna--Lions theory is based on the notion of renormalized solution.

\begin{definition}
	A weak bounded solution $v$ of the Cauchy problem (\ref{PDE}) is called renormalized solution,
	if for any smooth function $\beta:\mathbb{R}\to\mathbb{R}$ function $\beta(v)$ is again a weak solution.
\end{definition}

\subsection{Motion under the generalized dry friction control}

The idea behind the DiPerna--Lions theory is to construct a global phase flow, perhaps not everywhere uniquely defined, 
instead of a solution of an individual Cauchy problem for every initial condition. 
Our main result claims that in the phase space of system (\ref{sing_sys}) we can define a semiflow, 
which is continuous, uniquely defined everywhere, and gives a solution of the form (\ref{flow}) to the transport equation.

\begin{theorem}\label{motion2}
	There exists a continuous semiflow  $x\mapsto x(t)=\phi_t(x),\,t\geq0$ such that $v(x,t)=v(\phi_t(x))$ is the unique renormalized solution of the Cauchy problem for the transport equation
	\begin{equation}\label{transport}
		\frac{\partial v}{\partial t}=\left\langle Ax-B{\rm sign}\left\langle{B,\frac{\partial\rho}{\partial x}(x)}\right\rangle,\frac{\partial v}{\partial x}\right\rangle,\quad v(x,0)=v(x).
	\end{equation}
	Each curve $t\mapsto x(t)$ is absolutely continuous, and the differential inclusion (\ref{sing_sys}) holds.
\end{theorem}

\begin{proof}
We confine ourselves to a proof of existence of a continuous bounded solution of the transport equation (\ref{transport}), which is obtained as a limit of classical solutions of regularized equations. 
The remaining statements can be proved by using standard techniques from \cite{diperna,ovseev_diperna}.

We use a  two-parameter approximation of the problem. 
First, we choose a parameter $n\to\infty$ such that smooth convex functions $m_n:\mathbb{R}\to\mathbb{R}$ uniformly approximate the function $x\mapsto|x|$. 
Then, the derivative $s_n=m_n'$ approximates the function $\sign(x)$ in $L_1$. 
Note that $xs_n(x)\geq0$ for any $x\in\mathbb{R}$.
Second, we choose another parameter $\delta\downarrow0$, so as to freeze the motion under system (\ref{sing_sys}) 
within the $\delta$-neighborhood $U_\delta=\{\rho(x)\leq\delta\}$ of zero with respect to the distance $\rho$. 
In other words, we approximate ODE (\ref{sing_sys}) by the nonsingular equation
\begin{equation}\label{sing_sys2}
	\dot x=Ax- Bs_n\left(\left\langle B,\frac{\partial\rho}{\partial x}\right\rangle\right)
\end{equation}
in the domain $V_\delta=\{x\in\mathbb{R}^{2N}:\rho(x)\geq\delta\}$. 
It is important that all the neighborhoods $U_\delta$ are invariant under the phase flow of (\ref{sing_sys2}) for positive times because the radius-function $\rho$ is nonincreasing 
along the phase trajectories. 
Indeed, the following inequality holds:
$$
	\dot \rho=-s_n\left(\left\langle\frac{\partial {\rho}}{\partial x},B\right\rangle\right)\left\langle\frac{\partial {\rho}}{\partial x},B\right\rangle\leq0.
$$

There is a duality relation between the support function $H$ and the function $\rho$ (for details, see \cite{ovseev1,ovseev2}):
\begin{equation}\label{rho02}
	1=\rho\frac{\partial^2 {H}}{\partial p^2}\frac{\partial^2\rho}{\partial x^2}+\frac{\partial\rho}{\partial x}\otimes\frac{\partial {H}}{\partial p}.
\end{equation}
By using relation (\ref{rho02}), we rewrite equation (\ref{sing_sys2}) in the gradient form:
$$
	Bs_n\left(\left\langle{B,\frac{\partial\rho}{\partial x}}\right\rangle\right)=\rho\alpha(x)\frac{\partial}{\partial x}m_n\left(\left\langle{B,\frac{\partial\rho}{\partial x}}\right\rangle\right)+
	xs_n\left(\left\langle{B,\frac{\partial\rho}{\partial x}}\right\rangle\right)\left\langle{B,\frac{\partial\rho}{\partial x}}\right\rangle.
$$
The latter can be regarded as an approximation to
$$
	B\,{\rm sign}\,\left\langle{B,\frac{\partial\rho}{\partial x}}\right\rangle=\rho\alpha(x)\frac{\partial}{\partial x}\left|\left\langle{B,\frac{\partial\rho}{\partial x}}\right\rangle\right|+
	x\left|\left\langle{B,\frac{\partial\rho}{\partial x}}\right\rangle\right|, \quad \alpha(x)=\frac{\partial^2 H}{\partial p^2}, \quad H=H_\Omega.
$$
In particular, the ODE takes the form:
\begin{equation}
	\dot{x}=\left\{\begin{array}{ll}
	F(x)=f(x)-g(x)\frac{\partial}{\partial x}m_n\left(h(x)\right),&\mbox{ if }x \mbox{ is in $V_\delta$},\\
 	0,&\mbox{ if }x \mbox{ is in  $U_\delta$}.
	\end{array}\right.
\end{equation}
The functions involved
$$
	f(x)=Ax-xs_n\left(\left\langle{B,\frac{\partial\rho}{\partial x}}\right\rangle\right)\left\langle{B,\frac{\partial\rho}{\partial x}}\right\rangle,
	\quad
	g=\rho\alpha,\quad h=\left\langle{B,\frac{\partial\rho}{\partial x}}\right\rangle
$$
are rather smooth: they are locally Lipschitz outside zero. 
Eqs. (\ref{ODE}) approximates Eq. (\ref{sing_sys}) rewritten in the gradient form
$$
	\dot x=F(x)=f(x)-g(x)\frac{\partial}{\partial x}|h(x)|,
	\quad
	f(x)=Ax-x\left|\left\langle{B,\frac{\partial\rho}{\partial x}}\right\rangle\right|.
$$
It is important for us that the matrix $g=\rho\alpha$ is symmetric and nonnegative. 
By omitting the subscript $n$, we find that the corresponding transport equation takes the form
$$
	\frac{\partial v}{\partial t} =f_iv_i-g_{ij}h_jv_is(h)=F_iv_i,
$$
where
$$
	v_i=\frac{\partial}{\partial x_i}v,
	\quad
	h_i=\frac{\partial}{\partial x_i}h,
	\quad
	s(h)={\rm sign}\, h,
	\quad
	F_i=f_i-g_{ij}h_js(h),
$$
and the Einstein summation notation is used. 
By differentiation, we obtain the following equations for the vector-function $V$ with components $v_k$:
\begin{equation}\label{transport_ODE_k}
	\frac{\partial v_k}{\partial t} =F_iv_{k,i}+f_{i,k}v_i -g_{ij,k}h_iv_is(h)-g_{ij}h_{jk}v_is(h)-g_{ij}h_jh_kv_i\delta(h),
\end{equation}
where $v_{k,i}=\frac{\partial v_k}{\partial x_i}$, $h_{jk}=\frac{\partial^2 h}{\partial x_j\partial x_k}$,
$g_{ij,k}=\frac{\partial g_{ij}}{\partial x_k}$, and $\delta=\delta_n$ denotes $m_n^{\prime\prime}$. 
Equation (\ref{transport_ODE_k}) is again a transport equation with extra terms $f_{i,k}v_i-g_{ij,k}h_iv_is(h)-g_{ij}h_{ik}v_is(h)-g_{ij}h_ih_kv_i\delta(h)$ in the RHS. 
Fortunately, the most ``dangerous'' and singular term $\sigma_k=g_{ij}h_ih_kv_i\delta(h)$ has a positivity property:
$$
	g_{kl}v_l\sigma_k=g_{kl}h_kv_lg_{ij}h_jv_i\delta(h)=\left(\sum g_{kl}h_kv_l\right)^2\delta(h) \mbox{ is a positive measure.}
$$

All the other terms are linear functions of $V$ with coefficients bounded outside any neighborhood of zero. 
This implies that $w=\langle{gV,V\rangle}=g_{kl}v_lv_k$ is a kind of quadratic Lyapunov function:
\begin{equation}\label{Lyapunov_ODE_k}
	\frac{\partial w}{\partial t} \leq F_iw_{i}+LW,
	\quad
	W=|V|^2=\sum v_k^2.
\end{equation}
Here $L$ is a function uniformly bounded outside any neighborhood of zero. 
Since the matrix $g=\rho\alpha$ is not strictly positive definite, 
$W$ cannot be estimated via $w$, and Eq. (\ref{Lyapunov_ODE_k}) is insufficient for establishing an a priori bound for $w$, not to mention $W$. 
Nonetheless, we can use the estimate
\begin{equation}\label{bound_g}
	W=\sum v_k^2\leq C\left(\left(\sum x_kv_k\right)^2+\langle{gV,V}\rangle\right),
\end{equation}
where $C$ is a positive function bounded outside any neighborhood of zero. 
The bound holds because the kernel of the matrix $g(x)$ is the one-dimensional subspace of the phase space, generated by $x$. 
In view of inequality (\ref{bound_g}), we have to find an estimate for $z=\sum x_kv_k=Ev,$ where $E$ is the Euler operator $Ev=\sum{x_k}\frac{\partial v}{\partial x_k} $. 
By applying the Euler operator to the transport equation (\ref{transport_ODE_k}), we obtain:
\begin{equation}\label{transport_Euler}
	\frac{\partial z}{\partial t} =F_iEv_i+(EF_i)v_i=F_iz_i-F_iv_i+(EF_i)v_i.
\end{equation}
Here we use the commutation relation
$$
	\frac{\partial }{\partial x_i}E=E\frac{\partial }{\partial x_i}+\frac{\partial}{\partial x_i}
$$
which implies that $Ev_i=z_i-v_i$. It is easy to compute $EF_i$: The function
$$
	F(x)=Ax- Bs\left\langle{B,\frac{\partial\rho}{\partial x}}\right\rangle
$$
is clearly the sum of the homogeneous functions $Ax$ and $-Bs\left\langle{B,{\partial\rho}/{\partial x}}\right\rangle$ of degrees 1 and 0. 
Therefore, $EF_i$ is a locally bounded function. 
Relation (\ref{transport_Euler}) now implies that
\begin{equation}\label{transport_Euler2}
	\frac{\partial y}{\partial t}\leq F_iy_i+C'W,
\end{equation}
where $y=z^2$, and $C'$ is a locally bounded function. 
Inequality (\ref{bound_g}) says that $W\leq C\left(y+w\right)$.
Therefore, by summing inequalities (\ref{Lyapunov_ODE_k}) and (\ref{transport_Euler2}) we obtain that
$$
	\frac{\partial Y}{\partial t} \leq F_iY_{i}+MY,
	\quad
	Y=w+y,
$$
where the function $M$ is locally bounded outside zero uniformly wrt the scale $n$.

Inequality (\ref{Lyapunov_ODE_k}) is the crucial estimate that enables us to show that the flow 
$x\mapsto\Phi_t(x)=\Phi_{n,t}(x)$ corresponding to Eq. (\ref{ODE}) is locally Lipschitz, 
and besides the corresponding Lipschitz constant does not depend on the parameter $n$. 
Therefore, 
by passing to the limit $n\to\infty$ we conclude that there exists a Lipschitz limit of $\Phi_{n,t}$, which defines the semiflow $\phi_t(x)$ of Theorem \ref{motion2} within $V_\delta$. 
Since $\delta$ is arbitrary, this proves in particular that the map $x\mapsto\phi_t(x)$ is continuous if $x\neq0$ and $\phi_t(x)\neq0$.

It is in fact obvious that the map $x\mapsto\phi_t(x)$ is continuous at zero, because the flow $\phi$ maps any neighborhood $U_\delta$ of zero into itself. 
It remains to consider the case $x\neq0,\,\phi_t(x)=0$. 
Put $\tau=\inf\{t>0:\phi_t(x)=0\}$. It suffices to show that $\phi_\tau(y)$ is close to $\phi_\tau(x)=0$ if $y$ is sufficiently close to $x$. 
We know already that for any $\epsilon>0$ the point $\phi_{\tau-\epsilon}(x)$ depends on $x$ continuously. 
On the other hand, it is obvious that the map $t\mapsto\phi_t(y)$ is uniformly Lipschitz for $y$ in a neighborhood of $x$. 
Therefore,
\begin{equation}
	|\phi_\tau(y)-\phi_\tau(x)|\leq C|\epsilon|+|\phi_{\tau-\epsilon}(y)-\phi_{\tau-\epsilon}(x)|.
\end{equation}
Since $\epsilon$ is arbitrary and $|\phi_{\tau-\epsilon}(y)-\phi_{\tau-\epsilon}(x)|$ is arbitrarily small if $y$ is sufficiently close to $x$, the continuity is proved.
\end{proof}

A similar phenomenon was discovered by I.A. Bogaevskii \cite{bogaev} for the gradient differential equations $\dot{x}=-{\partial f}/{\partial x}$, where $f$ is a nonsmooth convex function.

\section*{Conclusion}

In our work, the control for damping a system of oscillators in the generalized dry friction form was studied. 
As it customarily happens in optimal control theory, this control gives rise to differential equations with discontinuous RHS. 
In the present work, 
we demonstrated that for the case under consideration the problem of existence and uniqueness of the motion under the control can be resolved in the framework of the DiPerna--Lions theory. 
It seems important to study similar problems for the optimal control.

An interesting development of the considered problem of damping a system of oscillators is given by passage to the infinite-dimensional case. 
For example, the problem of damping a closed string under bounded load applied to a fixed point leads to nontrivial issues closely related to the above discussed ones.

\section*{Acknowledgements}

This work is supported by the Russian Foundation for Basic Research (projects 14-08-00606 and 14-01-00476).

\end{document}